\documentclass[12pt,reqno]{amsart}

\usepackage{hyperref}
\usepackage{amsthm, amsmath, amscd, amssymb,centernot}
\usepackage{amsfonts}
\usepackage[all]{xy}
\usepackage{comment}
\usepackage[T1]{fontenc}
\usepackage[left=2.5cm,top=2.5cm,bottom=3cm,right=2.5cm]{geometry}
\usepackage{graphicx}
\usepackage{tikz-cd}
\usepackage{epsf}
\usepackage{float}

\makeatletter
\@namedef{subjclassname@2020}{%
\textup{2020} Mathematics Subject Classification}
\makeatother

\newcommand{\proj}{\mathbb{P}}

\theoremstyle{plain}
\numberwithin{equation}{section}
\newtheorem{theorem}{Theorem}[section]
\newtheorem*{theorem*}{Theorem}

\newtheorem{corollary}[theorem]{Corollary}

\theoremstyle{definition}
\newtheorem{question}[theorem]{Question}
\newtheorem{definition}[theorem]{Definition}
\newtheorem{remark}[theorem]{Remark}
\newtheorem{example}[theorem]{Example}

\newtheorem{assumption}[theorem]{Assumption}
\usepackage{xcolor}
\usepackage{comment}

\begin{document}
	\title[Seshadri constants  of curve configurations on surfaces]{Seshadri constants  of curve configurations on surfaces}
	\author[K. Hanumanthu]{Krishna Hanumanthu}

\address{Chennai Mathematical Institute, H1 SIPCOT IT Park, Siruseri, Kelambakkam 603103,
India.}

\email{krishna@cmi.ac.in}

\author[P. K. Roy]{Praveen Kumar Roy}

\address{UM-DAE Centre for Excellence in Basic Sciences, University of Mumbai Santacruz, Mumbai 400098, India}

\email{praveen.roy@cbs.ac.in}

\author[A. Subramaniam]{Aditya Subramaniam}

\address{Indian Institute of Science Education and Research Tirupati, Rami Reddy Nagar, Karakambadi Road, Mangalam (P.O.),
Tirupati, Andhra Pradesh, INDIA – 517507.}

\email{adityasubramaniam@labs.iisertirupati.ac.in}


\date{\today}

	\begin{abstract}
	Let $X$ be a complex nonsingular projective surface and let $L$ be an ample line bundle on $X$. We study multi-point Seshadri constants of $L$ at singular points of certain connected arrangements of curves on $X$. We pose some questions about such Seshadri constants and prove results in many cases, including star arrangements of curves. We also study the configurational Seshadri constants for curve arrangements on surfaces and compare them with the usual Seshadri constants. We give several examples illustrating the properties that we study. 
	\end{abstract}
	
		\maketitle

	\section{Introduction}
	In this note, we study multi-point Seshadri constants of ample line bundles centered at the singular loci of certain curve arrangements on surfaces. We recall the notion of  multi-point Seshadri constants briefly below. 
	
	Let $X$ be a smooth complex projective variety and let $L$ be a nef line bundle on
	$X$. Let $r \ge 1$ be an integer and let $x_1,\ldots,x_r$ be distinct
	points of $X$.
	The {\it Seshadri constant} of $L$ at $x_1,\ldots,x_r \in X$ is defined as: 
	$$\varepsilon(X,L,x_1,\ldots,x_r):=  \inf\limits_{\substack{C \subset
			X{\rm ~a~ curve~ with} \\C \cap \{x_1,\ldots,x_r\}
			\ne \emptyset}} \frac{L\cdot C}{\sum\limits_{i=1}^r {\rm
			mult}_{x_i}C}.$$ 
	
	It is easy to see that the infimum above is the same as the infimum taken over
	irreducible, reduced curves $C$ such that $C \cap \{x_1,\ldots,x_r\}
	\ne \emptyset$.
	
	The following is a  well-known upper bound for Seshadri
	constants. 
	Let $n$ be the dimension of $X$. Then for
	any $x_1,\ldots,x_r \in X$, 
	\begin{eqnarray*}\label{wellknown}
		\varepsilon(X,L,x_1,\ldots,x_r) \le
		\sqrt[n]{\frac{L^n}{r}}.
	\end{eqnarray*}
	The study of multi-point Seshadri constants is interesting even in the case of the projective plane $\mathbb{P}^2.$ A famous conjecture of Nagata asserts that if $r \geq 10$ and $x_1,\cdots, x_r \in \mathbb{P}^2 $ are very general, then 
	$$\varepsilon(\mathbb{P}^2,\mathcal{O}_{\mathbb{P}^2}(1),x_1,\ldots,x_r)= \frac{1}{\sqrt{r}}.$$ 
	Nagata \cite{Nag} proved this when $r=s^2$ for an integer $s$, but it is open in all other cases. 
	
	On the other hand, when the points $x_1, \cdots, x_r$ lie in special position, the corresponding  Seshadri constant is frequently rational, and it is a very interesting problem to compute it. For example, if the points are collinear then 
	$\varepsilon(\mathbb{P}^2,\mathcal{O}_{\mathbb{P}^2}(1),x_1,\ldots,x_r) = \frac{1}{r}$. 
	
	In this direction, the author in \cite{Pok} considers behaviour of multi-point Seshadri constants  
	of $\mathcal{O}_{\mathbb{P}^2}(1)$ at the 
	set of singular points  of a line configuration in the plane $\mathbb{P}^2.$  The following 
	question was posed in \cite{Pok}. 
	
		\begin{question}\label{qn1}
		Let $x_1, \cdots, x_r \in \mathbb{P}^2$ be the set of singular points in a configuration of lines  which is not a pencil of lines. Let $k$ denote the maximum number of points which lie on a line. Is it true that 
		$$\varepsilon(\mathbb{P}^2,\mathcal{O}_{\mathbb{P}^2}(1),x_1,\ldots,x_r)= \frac{1}{k} ? $$
	\end{question}
	In \cite{Pok}, the author gives specific examples like line arrangements satisfying \textit{Hirzebruch's property} and star arrangement of $d$ lines in $\mathbb{P}^2$ where there is an affirmative answer to the above question. 
	In \cite{JP}, the authors extend the above study to the case of arrangements of plane curves of a fixed degree. These arrangements were introduced in the context of Harbourne constants in \cite{PRS} and the bounded negativity conjecture. In order to study the  local negativity phenomenon for algebraic surfaces, it was more reasonable to consider curve arrangements instead of irreducible curves as they are more difficult to construct.  
	
    	In \cite{JP}, the authors also study the multi-point Seshadri constants of $\mathcal{O}_{\mathbb{P}^{2}}(1)$ centered at singular loci of certain curve arrangements $\mathcal{C}$ in $\mathbb{P}^{2}$ using a combinatorial invariant called the  \textit{configurational Seshadri constant of $\mathcal{C}$}. They give lower bounds for the configurational Seshadri constants of $\mathcal{O}_{\mathbb{P}^{2}}(1)$ and also provide some actual values of the multi-point Seshadri constants for some classes of curve arrangements, comparing them with their associated configurational Seshadri constants.
	
In this paper, we continue this study by looking at connected curve arrangements on arbitrary surfaces. We prove some analogues of the results in \cite{JP} for arbitrary surfaces. We work over the field of complex numbers.

We recall below the basic definitions that we require in this paper.

	\begin{definition}[Transversal arrangement]
		Let $\mathcal{C}=\{C_1,C_2, \ldots ,C_d\}$ be an arrangement of curves
		on a smooth projective surface $X$ such that their union $C_1 \cup C_2 \dots \cup C_d$ is a connected curve in $X$. 
		We say that $\mathcal{C}$ is a \emph{transversal arrangement}
		if $d\geq 2$, all  curves $C_i$ are smooth, irreducible 
		and they intersect pairwise transversally.
	\end{definition}
	
	Let $\text{Sing}(\mathcal{C})$ be the set of all intersection points
	of the curves in a transversal arrangement $\mathcal{C}$. 
	Let $s$ denote the number of points in $\text{Sing}(\mathcal{C})$.

	\begin{definition}[Combinatorial invariants of transversal arrangements]
		Let	$\mathcal{C}$ be a transversal arrangement on a smooth surface
		$X$. For a point $p\in X,$
		let $r_p$ denote the  number of elements of $\mathcal{C}$ that pass
		through $p$. We call $r_p$ the \textit{multiplicity} of $p$ in
		$\mathcal{C}.$ We say $p$ is a $k$-fold point of $\mathcal{C}$ 
		if there are exactly $k$ curves in $\mathcal{C}$ passing through $p.$
		For a positive integer $k\geq2$, $t_k$ denotes the number of $k$-fold points in $\mathcal{C}$.
	\end{definition}
	These numbers satisfy the following standard equality, which follows by
	counting incidences in a transversal arrangement in two ways:
	\begin{equation}\label{eq:combinatorial general}
		\sum_{i<j}(C_i\cdot C_j)=\sum_{k\geq 2}\binom{k}{2}t_k.
	\end{equation}
	Also, let $$f_i=f_i(\mathcal C) :=\sum_{k\geq 2}k^i t_k.$$
	In particular, $f_0=s$ is the number of points in
	$\text{Sing}(\mathcal{C})$. 
	
	\begin{definition}
		Let $\mathcal{C}=\{C_1,C_2, \ldots ,C_d\}$ be a connected arrangement of irreducible curves
		on a smooth projective surface $X$ and let $\text{Sing}(\mathcal{C})$ denote the singular locus.  Let $L$ be an ample line bundle on $X$. We define the \textit{configurational Seshadri constant} of $\mathcal{C}$ as 
			$$\varepsilon_\mathcal{C}(L):= \frac{\sum\limits_{i=1}^{d}L\cdot C_i}{\sum\limits_{p \in \text{Sing}(\mathcal{C}) }\text{mult}_p(\mathcal{C})}.$$
	\end{definition}	
	

One class of arrangements studied in this paper are 
 natural analogues of the classical star arrangements of lines on $\mathbb{P}^2$. 
We make the following definition. 

\begin{definition} Let $X$ be a nonsingular projective surface.  
A transversal arrangement of curves $\mathcal{C}$ on $X$
		 is called a \emph{star  configuration} if the only intersection points are double points.
	\end{definition}
	
If $\mathcal{C}= \{C_1,C_2, \ldots ,C_d\}$ is a star arrangement and $C_i$ are linearly equivalent to each other then $$t_2=\frac{C_1^2d(d-1)}{2}\text{ and }t_k = 0 \text{ for } k > 2.$$ 

Star configurations of lines in $\mathbb{P}^2$, and more generally of hyperplanes in $\mathbb{P}^n$, are extensively studied; 
see \cite{CV,GHM, JP}, for some results. 	
	
In Section \ref{main}, we prove one of our main results (Theorem \ref{main-theorem})
for a transversal arrangement $\mathcal{C}$ of curves on a smooth projective surface under a nefness condition. We will then assume that all curves in $\mathcal{C}$ are linearly equivalent to a fixed divisor and prove some results for such configurations using Theorem \ref{main-theorem}. The main result in this set-up is Corollary \ref{prop:star}.  

In Section \ref{ruled}, we consider some examples on ruled surfaces. In Section \ref{compare}, we compare the usual Seshadri constants with configurational Seshadri constants. In Section \ref{bounds}, we give lower bounds for configurational Seshadri constants for curve arrangements on surfaces. 

Finally, we give an example (Example \ref{K3}) of a computation of the multi-point Seshadri constant at the singular locus of an arrangement of curves on a K3 surface. In this example, the nefness hypothesis of Theorem \ref{main-theorem} is not true but the conclusion about the Seshadri constant still holds. 
	
	\section{Seshadri constants for certain transversal curve arrangements}\label{main}	

In this section, we will first prove a result computing Seshadri constants for arbitrary transversal arrangements under a nefness condition. We will then study special cases of this theorem.

We first introduce some notation. 
		Let $\mathcal{C}$ be a transversal arrangement of curves on a smooth projective surface $X$. 
  
  For each $1\le i \le d$, let $b_i =  \# \, C_{i} \cap {\rm Sing}(\mathcal{C})$. Since $\mathcal{C}$ is connected, $b_i >0$ for all $i$.
   We define the \emph{base constant} of $\mathcal{C}$ as 
   $$\text{bs}(\mathcal{C}) = \max\{b_1,\dots, b_d\},$$ i.e., 
$\text{bs}(\mathcal{C})$		
  is equal to the maximum number of singular points  of $\mathcal{C}$ that are contained in a single curve $C_{i} \in \mathcal{C}$.

The following question is a generalization of Question \ref{qn1} to curve arrangements on arbitrary surfaces.

	\begin{question}\label{question}
		Let $\mathcal{C} =\{C_1,C_2, \ldots ,C_d\}$ be a transversal arrangement of curves on a smooth projective surface $X$. Let $L$ be an ample line bundle on $X$. Is it true that
		\begin{equation*}
  \varepsilon(X, L,{\rm Sing}(\mathcal{C})) =  \min_{1\le i \le d}\left\{\frac{L\cdot C_i}{b_i}\right\}?
		\end{equation*}
	\end{question}

By \cite[Example 3.5]{JP}, the above question has a negative answer for the Hesse arrangement of conics 
in $\mathbb{P}^2$; see Example \ref{hesse}. However it is open for line arrangements in $\mathbb{P}^2$. 
It is interesting to study the situations in which it is true on arbitrary surfaces. Our main theorem below answers this question positively under an additional nefness condition.


	
 
	\begin{theorem}
		\label{main-theorem} Let $X$ be a nonsingular projective surface and let $\mathcal{C} = \{C_{1}, \ldots, C_{d}\}$ be a transversal configuration of curves on $X$ with $d \geq 4$. 
		Let $L$ be an ample line bundle on $X$
		such that  $\frac{d\sum b_i}{(d-1)^2(\sum L\cdot C_i)}L-C_i$ is nef for each $i$.   Then 
		$$\varepsilon(X, L,{\rm Sing}(\mathcal{C})) = \min_{1\le i\le d}
  {\left\{\frac{L\cdot C_i}{b_i}\right\}}.$$
	\end{theorem}
	\begin{proof}
	 Let $\mathcal{P} = \{p_{1}, \ldots, p_{s}\}$ be the singular locus of  $\mathcal{C}$. For each $i$, exactly $b_i$ points in $\mathcal{P}$  lie on $C_i$. Since each $C_i$ is smooth, the Seshadri ratio computed by $C_i$ is precisely 
	  $ \frac{L\cdot C_i}{b_i}$. 

  We will now show that the Seshadri ratio given by any curve $D$ not in $\mathcal{C}$ is at least $\min_{1\le i\le d}
  {\left\{\frac{L\cdot C_i}{b_i}\right\}}$. 
  
	 Suppose that there exists an irreducible and reduced curve $D\notin \mathcal{C}$ having multiplicity $m_{p}(D)$ at each point $p \in \mathcal{P}$ such that $D \cap \mathcal{P} \ne \emptyset$ and 
		$$\frac{L \cdot D}{\sum_{p \in \mathcal{P}} m_{p}(D)} < \min_{1\le i \le d}\left\{\frac{L\cdot C_i}{b_i}\right\}.$$
So for each $1 \le i \le d$, we have 
                $$b_i(L \cdot D)  < (L\cdot C_i) \sum_{p \in \mathcal{P}} m_{p}(D).$$ 
                
Adding over all $i$, we have 
		$$(\triangle) : \quad (L \cdot D)\sum_{i=1}^d b_i < \left(\sum_{p \in \mathcal{P}} m_{p}(D)\right)\sum_{i=1}^d L\cdot C_i.$$

  By B\'ezout's theorem applied to $D$ and $C_1+\dots+C_d$, we obtain
		\begin{eqnarray*}
		 D\cdot (C_{1}+ \ldots + C_{d}) &\geq& \sum_{p \in \mathcal{P}}m_{p}(D)\cdot m_{p}(C_{1}+ \ldots + C_{d})\\
		&\stackrel{(*)}{\geq}& 2\sum_{p \in \mathcal{P}} m_{p}(D) \\
		&\stackrel{(\triangle)}{>}& \frac{2(L \cdot D)\sum b_i}{\sum L\cdot C_i},
		\end{eqnarray*}
		where $(*)$ comes from the fact that all the singular points of $C_{1}+ \ldots + C_{d}$ are at least double points. Now we use the nefness 
		of $\frac{d\sum b_i}{(d-1)^2(\sum L\cdot C_i)}L-C_i$ and note the following  inequality: 
		\[
		\frac{d\sum b_i}{(d-1)^2(\sum L\cdot C_i)}L\cdot D \geq D\cdot C_i.
		\] 
  
		After adding over all $i=1,...,d$, this leads to the following, 
		\begin{eqnarray*}
			\frac{d^2\sum b_i}{(d-1)^2(\sum L\cdot C_i)}L\cdot D &\geq&  D\cdot (C_{1}+ \ldots + C_{d}) > \frac{2(L \cdot D)\sum b_i}{\sum L\cdot C_i} \\
			\Rightarrow \left(\frac{d}{d-1} \right)^2 &>& 2 \quad \quad \text{(since $\sum b_i > 0$ and $L$ is ample)}.
		\end{eqnarray*}
		This is a contradiction, since $d\geq 4$.
	\end{proof}

\begin{remark}
Several authors have studied multi-point Seshadri constants on surfaces; for a sample of results, see \cite{B,HM, HR, PR, RR, SS}. 
Many of these results consider general points. 
Theorem \ref{main-theorem} gives a computation of Seshadri constants at \textit{special} points. 
\end{remark}

Now we will work with configurations satisfying some additional assumptions. Specifically, we will make the following assumption. 

	\begin{assumption}\label{star1}
		Let $\mathcal{C}=\{C_1,C_2, \ldots ,C_d\}$, $d\geq 4$ be a (connected) 
  transversal arrangement of curves 
		on a complex smooth projective surface $X$ such that all the curves 
  $C_i$ in $\mathcal{C}$ are linearly equivalent to a
			fixed divisor $A$ on $X$.
	\end{assumption}

 \begin{remark} 
Recall that the arrangements we study in this paper are all connected. Since each curve in $\mathcal{C}$ is assumed to be linearly equivalent to a fixed divisor $A$, 
Assumption \ref{star1} implies in particular that $C_i\cdot C_j > 0$ for all $1\le i,j \le d$.  
 \end{remark}

For an arbitrary curve arrangement $\mathcal{C}$ satisfying Assumption \ref{star1}, the base constant is at most 
$C_1^2(d-1)$. So the following question arises naturally from
Question \ref{question}. 

	\begin{question}\label{QuestM}
		Let $\mathcal{C} =\{C_1,C_2, \ldots ,C_d\}$ be a transversal arrangement of curves on a smooth projective surface $X$
		satisfying Assumption \ref{star1}. Let  ${\rm Sing}(\mathcal{C})$ denote the singular locus of $\mathcal{C}$ and let $L$ be an ample line bundle on $X$. Is it true that
		\begin{equation}\label{ques}
			\varepsilon(X, L,{\rm Sing}(\mathcal{C})) \geq \frac{L \cdot C_1}{C_1^2(d-1)} ?	
		\end{equation}	
		
	\end{question}	
	
	For star arrangements, the base constant is exactly equal to $C_1^2(d-1)$. We answer the above question positively for some arrangements below.

	
	\begin{corollary}
		\label{prop:star} Let $X$ be a nonsingular projective surface and let $\mathcal{C} = \{C_{1}, \ldots, C_{d}\}$ be a star configuration of curves on $X$ with $d \geq 4$ such that each $ C_i$ is linearly equivalent to a fixed divisor $A$ on $X$. Let $L$ be an ample line bundle on $X$
		such that  $\frac{dC_1^2}{(d-1)(L\cdot C_1)}L-C_1$ is nef.   Then 
		$$\varepsilon(X, L,{\rm Sing}(\mathcal{C})) = \frac{L\cdot C_1}{C_1^2(d-1)}.$$
	\end{corollary}

 \begin{proof}
Under the given hypotheses, we have $b_i = C_1^2(d-1)$  and $L\cdot C_i = L\cdot C_1$
for all $i$. 
The proof is now immediate from Theorem \ref{main-theorem}.    
 \end{proof} This immediately gives the following.

	\begin{corollary}
		\label{coro:star}
		If $\mathcal{C} = \{C_{1}, \ldots, C_{d}\}$ is a star configuration on a nonsingular projective surface $X$ with $d \geq 4$  and $L$ is an ample line bundle  on $X$ such that $ C_i \in |mL|$ for all $i$ and a positive integer $m>0$, then 
		$$\varepsilon(X, L,{\rm Sing}(\mathcal{C})) = \frac{1}{m(d-1)}.$$
	\end{corollary}

The proof of Theorem \ref{main-theorem} can be modified to show that for arrangements
satisfying Assumption \ref{star1}, which are not necessarily star arrangements, we still have the following lower and upper bounds for the Seshadri constants.

\begin{corollary}\label{cor-main}
Let $X$ be a nonsingular projective surface and let $\mathcal{C} = \{C_{1}, \ldots, C_{d}\}$ be a configuration of curves on $X$ satisfying Assumption \ref{star1}.
 Let $L$ be an ample line bundle on $X$
		such that  $\frac{dC_1^2}{(d-1)(L\cdot C_1)}L-C_1$ is nef.   Then
$$\frac{L\cdot C_1}{C_1^2(d-1)} \le \varepsilon(X, L,{\rm Sing}(\mathcal{C})) \le \frac{L \cdot C_1}{{\rm bs} (\mathcal{C})}  .$$
\end{corollary} 
\begin{proof}

Let  $D \notin \mathcal{C}$ be 
a reduced and irreducible curve which contains some points of the singular locus $\mathcal{P}$ of $\mathcal{C}$. Essentially following the proof of Theorem \ref{main-theorem}, we can show that 
$$ \frac{L\cdot C_1}{C_1^2(d-1)} \le \frac{L \cdot D}{\sum_{p \in \mathcal{P}} m_{p}(D)}.$$


Indeed, suppose that the above inequality does not hold. 
As in the proof of Theorem \ref{main-theorem}, B\'ezout's theorem applied to $D$ and $C_1+\dots+C_d$ gives

$$		 D\cdot (C_{1}+ \ldots + C_{d}) \ge  
\frac{2(d-1)(L \cdot D)C_1^2}{L\cdot C_1}.$$

		Now we use the nefness of 
  $\frac{dC_1^2}{(d-1)(L\cdot C_1)}L-C_1$ and obtain a contradiction exactly as in the proof of Theorem \ref{main-theorem}. We omit the details since the argument is identical.

For curves in $\mathcal{C}$, the least Seshadri ratio is  $\frac{L \cdot C_1}{\rm{bs} (\mathcal{C})}$. 
Since we always have ${\rm bs} (\mathcal{C}) \le C_1^2(d-1)$, the proof is complete. 
\end{proof}

\begin{example}\label{hesse}
Consider the Hesse arrangement $\mathcal{C} = \{C_1, \ldots, C_{12}\}$  of conics in $\mathbb{P}^2$
which has 21 singular points.  The base constant of $\mathcal{C}$  is 8.  
 The Seshadri constant of $L = \mathcal{O}_{\mathbb{P}^2}(1)$ at these 21 points is known to be $\frac{1}{5}$; see \cite[Example 3.5]{JP}. 

We note that the nefness assumption of Theorem \ref{main-theorem} does not hold here. Indeed, $b_i = 8$ for each $1 \le i \le 12$. So the divisor 
$\left(\frac{12(b_1+\dots + b_{12})}{11*11*24}\right)L - C_1$ is not nef because 
$b_1+\dots + b_{12} < 4*11^2$. 

However, Corollary \ref{cor-main} does apply in this case. 
It is easy to see that $\frac{24}{11}L - C_1$ is nef. 
Then Corollary \ref{cor-main} implies
$$\frac{1}{22} <  \varepsilon(\mathbb{P}^2, L,{\rm Sing}(\mathcal{C})) = \frac{1}{5}<  \frac{1}{4}.$$

This shows that the inequalities in Corollary \ref{cor-main} are strict, in general. In particular,  Corollary \ref{prop:star} does not hold in this case. Note that $\mathcal{C}$ is not a star arrangement.

\end{example}

\begin{example}
		Let $\mathcal{L}_n \subset \mathbb{P}^2$ denote the $n$-th Fermat arrangement of $3n$ lines $l_i$ with $n \geq 3$ given by the linear factors of the polynomial 
		\[
		Q(x,y,z)=(x^n-y^n)(y^n-z^n)(z^n-x^n).
		\]
		It is well known that it has $n^2$ (say $\{p_1,\ldots,p_{n^2}\}$) triple points and $3$ points of multiplicity $n$ (say $\{q_1,q_2,q_3\}$); see \cite[Example II.6]{U}.  We now consider a 
		curve $C$ of degree $6$ not passing through either $p_i$ or $q_j$ and take the double cover 
				$\pi : X \longrightarrow \mathbb{P}^2$
		branched along $C$. Then $X$ is a K3 surface  and $L : = \pi^*(\mathcal{O}(1))$ is ample. 

For each $i$, let $C_i$ denote the inverse image of $l_i$ and let 
$\mathcal{C}_n := \{C_1,\ldots, C_{3n}\}$ be an arrangement of curves on $X$. Note that $C_i^2 = 2$ for every $i$ and 
$C_i$ is linearly equivalent to $C_j$ for all $i,j$. 
One can check that the base constant of the arrangement $\mathcal{C}_n$ is $2n+2$. 

It is easy to see that this arrangement satisfies the hypotheses of Corollary \ref{cor-main}. 
So we have the following inequalities:
$$\frac{1}{3n-1} \le \varepsilon(X, L,{\rm Sing}(\mathcal{C}_n)) \le \frac{1}{n+1}.$$


\end{example}

	We now construct examples such that the conclusion of Corollary \ref{prop:star} holds, even without the nefness condition in the hypothesis. We construct these examples by obtaining transversal arrangement of curves on rational ruled surfaces coming from line arrangements in $\mathbb{P}^2.$ 

\subsection{Examples on ruled surfaces}\label{ruled}	
	
	We first recall some basic facts about ruled surfaces. We follow the notation of \cite[Chapter V, Section 2]{Har}.
	
	Let $C$ be a smooth complex curve of genus $g$ and let $\pi: X \to C$ be a ruled surface over $C$. We choose a normalized vector bundle $E$ of rank 2 on $C$ such that $X = \proj_C(E)$. Let $e :=-\text{deg}(E)$. Let $C_0$ be the image of a section of $\pi$ such that $C_0^2 = -e$ and let $f$ be a fiber of $\pi$. Then the Picard group of $X$ modulo numerical equivalence is a free abelian group of rank 2 generated by $C_0$ and $f$. We have $C_0\cdot f=1$ and $f^2 =0$.  
	
	A complete characterization of ample line bundles and irreducible curves on $X$ is known. For this, and other details on ruled surfaces, we refer to \cite[Chapter V, Section 2]{Har}.

	Let $X = X_e \to \mathbb{P}^1$ be a rational ruled surface with invariant $e \ge 1$. 
	Given a line arrangement in $\proj^2$, one can obtain an arrangement of
	curves on $X_e$, following a construction outlined in \cite[Example
	15]{E}, where a specific finite morphism $X_e \to X_1$ of degree $e$ is described. Note
	that $X_1$ is isomorphic to a blow up of $\proj^2$ at a point. So we
	can pull-back lines in $\proj^2$ to $X_e$ which are in the class 
	$(1,e)$. If $\mathcal{L}$ is a line arrangement of $d$ lines in the plane, its
	pull-back gives a curve arrangement $\mathcal{C}$ of $d$ curves in
	$X_e$. 
	
	To be more precise, suppose that $\mathcal{L}$ has $s$ singularities 
	and $t_k$ denotes the number of singular points of $\mathcal{L}$ of
	multiplicity $k$. Then the singular points of $\mathcal{C}$ are
	precisely the pre-images of singularities of $\mathcal{L}$. So 
	$\mathcal{C}$  has $es$ singular points and the
	number of singular points of multiplicity $k$ is $et_k$. Note that
	each curve in $\mathcal{C}$ is in the class $(1,e)$ and has
	self-intersection $e$.
	
	In the next example, the conclusion of Corollary \ref{prop:star} holds, even without the nefness condition in the hypothesis. 
	\begin{example}\label{nef}
		Consider a transversal arrangement $\mathcal{C} = \{C_{1}, \ldots, C_{5}\}$ of curves on a rational ruled surface $X=X_e$ with invariant $e=2$ such that $ C_i \in |C_0+2f|$ and such that the only intersection points are ordinary double points. This configuration can be obtained as pullback of a star configuration of lines from $\mathbb{P}^2$, using the above construction. Note that each $C_i$ contains exactly $(5-1)e=8$ double points from the arrangement.
		
 Let $L=C_0+3f$ be an ample line bundle on $X.$ Then we have $$(5C_1^2/4L\cdot C_1))L-C_1= \frac{1}{6}(-C_0+3f),$$ which is not nef as $(-C_0+3f)\cdot f = -1 < 0.$ Hence the nefness hypothesis in Corollary \ref{prop:star} fails. 
		
		We now verify that 
		$$\varepsilon(X, L,{\rm Sing}(\mathcal{C})) = \frac{L\cdot C_1}{4C_1^2}=\frac{3}{8}.$$
		
		Let $\mathcal{P} = \{p_{1}, \ldots, p_{20}\}$ denote the set of all the singular points of $\mathcal{C}$. Suppose that there exists an irreducible and reduced curve $D$, different from each $C_{i}$ for $i \in \{1,\ldots, 5\}$, having multiplicity $m_{p}(D)$ at each point $p \in \mathcal{P}$ such that
		\begin{equation}\label{eg ineq}
			\frac{L \cdot D}{\sum_{p \in \mathcal{P}} m_{p}(D)} < \frac{3}{8}.
		\end{equation}
		Using B\'ezout's theorem, we have
		\begin{eqnarray*}
		D\cdot (C_{1}+ \ldots + C_{5})=5(D\cdot C_1) &\geq& \sum_{p \in \mathcal{P}}m_{p}(D)\cdot m_{p}(C_1+\ldots+C_5)\\&=& 2\sum_{p \in \mathcal{P}} m_{p}(D)\\&>&\frac{16(L \cdot D)}{3}.\end{eqnarray*}
		
		Since $C_1$ is linearly equivalent to $L-f$ and $f$ is nef,  we have $$5(D \cdot L) \geq 5(D \cdot C_{1})> \frac{16(L \cdot D)}{3} ,$$ which is absurd since $L$ is ample.
		
	\end{example}
	
	\begin{example}\label{Hirz}
		We now construct examples where the inequality in Equation \ref{ques} is a strict inequality. Consider a transversal arrangement $\mathcal{C}_n = \{C_{1}, \ldots, C_{3n}\}$ of curves on a rational ruled surface $X=X_e$ with invariant $e\ge n>1$ such that $ C_i \in |C_0+ef|$ and such that the curves $C_i$ come as a pullback of a line arrangement $\mathcal{L} \subset \mathbb{P}^2$ satisfying Hirzebruch's property and having singular points with multiplicity greater than two, using the above construction, for more details, see \cite{Pok}. Recall that a line arrangement $\mathcal{L} \subset \mathbb{P}^2$ satisfies  Hirzebruch's property if the number of lines is equal to $3n$ for some $n \in \mathbb{Z}_{>0}$ and each line from $\mathcal{L}$ intersects any other line at exactly $n+1$ points. 
		
		Let $L=C_0+(e+1)f$ be an ample line bundle on $X$. We now verify that  
		$$\varepsilon(X, L,{\rm Sing}(\mathcal{C}_n)) = \frac{L\cdot C_1}{e(n+1)}=\frac{e+1}{e(n+1)}.$$
		Note that $\frac{e+1}{e(n+1)} > \frac{e+1}{e(3n-1)}$, which shows that the 
		inequality in Equation \ref{ques} is strict.

		Let $\mathcal{P}$ denote the set of all intersection points of $\mathcal{C}_n$. 
		The Seshadri ratio computed by $C_1$ is precisely $\frac{e+1}{e(n+1)}$. 
		Suppose that there exists an irreducible and reduced curve $D$, different from each $C_{i}$ for $i \in \{1,\ldots, 3n\}$, having multiplicity $m_{p}(D)$ at each point $p \in \mathcal{P}$ such that
		\begin{equation}\label{eg ineq}
			\frac{L \cdot D}{\sum_{p \in \mathcal{P}} m_{p}(D)} < \frac{e+1}{e(n+1)}.
		\end{equation}
		
		Using B\'ezout's theorem, we have
		\begin{eqnarray*}D\cdot (C_{1}+ \ldots + C_{3n})=3n(D\cdot C_1) &\geq& \sum_{p \in \mathcal{P}}m_{p}(D)\cdot m_{p}(C_1+\ldots+C_{3n}) \\ &\geq &3\sum_{p \in \mathcal{P}} m_{p}(D)\\ &>&\frac{3(L \cdot D)e(n+1)}{e+1}.\end{eqnarray*}
		
		Since $C_1$ is linearly equivalent to $L-f$ and $f$ is nef,  we have $$3n(D \cdot L) \geq 3n(D \cdot C_{1})>\frac{3(L \cdot D)e(n+1)}{e+1},$$ which is not possible since $e \ge n$ by the choice of $e$. 
		
	\end{example}
	
	\begin{example}\label{KCW}
		There are examples of line arrangements in $\mathbb{P}^2$ satisfying Hirzebruch's property and such that all the singular points have multiplicity greater than two; for more details, see \cite{Pok}. 
		
		First, we consider the $n$-th Fermat arrangement $\mathcal{L}_n $ of lines  in $\mathbb{P}^2$  with $n \geq 3$. 
		As noted above, this arrangement consists of $3n$ lines with $t_n=3$ and $t_3=n^2.$ Note that $\mathcal{L}_n$ satisfies Hirzebruch property. Using the above construction, we can pullback $\mathcal{L}_n$ to a curve arrangement  $\mathcal{C}_n'$ on $X_e$ with $e > n$. Then for the ample line bundle $L=C_0+(e+1)f$ on $X_e$ as in Example \ref{Hirz}, we get
		$$\varepsilon(X_e, L,{\rm Sing}(\mathcal{C}_n')) = \frac{L\cdot C_1}{e(n+1)}=\frac{e+1}{e(n+1)}.$$
		
		Next,  let $\mathcal{K} $ denote the Klein arrangement of lines consisting of 21 lines in $\mathbb{P}^2$ with $t_4=21$ and $t_3=28.$ Note that $\mathcal{K}$ satisfies Hirzebruch property. Using the above construction, we can pullback $\mathcal{K}$ to a curve arrangement  $\mathcal{K}'$ on $X_e$ with $e > n$ as in the above construction. Then for the ample line bundle $L$ on $X_e$ as in Example \ref{Hirz}, we get
		$$\varepsilon(X_e, L,{\rm Sing}(\mathcal{K}')) = \frac{L\cdot C_1}{e(n+1)}=\frac{e+1}{8e}.$$
		
		We now consider another arrangement satisfying Hirzebruch property, namely the Wiman arrangement of lines $\mathcal{W}$ in $\mathbb{P}^2$ consisting of 45 lines with $t_3=120$, $t_4=45$ and $t_5=36.$ Using the above construction, we can pullback $\mathcal{W}$ to a curve arrangement  $\mathcal{W}'$ on $X_e$ with $e > n$ as in the above construction. Then for the ample line bundle $L$ on $X_e$ as in Example \ref{Hirz}, we get
		$$\varepsilon(X_e, L,{\rm Sing}(\mathcal{W}')) = \frac{L\cdot C_1}{e(n+1)}=\frac{e+1}{16e}.$$
	\end{example}

	\subsection{Comparing the Seshadri constant and the configurational Seshadri constant} \label{compare}
	
	We now study possible discrepancies between the configurational and the multi-point Seshadri constants with some examples. It is clear that for an arrangement $\mathcal{C}$ on a smooth projective surface $X$ and an ample line bundle $L,$ we always have the following inequality: 
	$$\varepsilon_\mathcal{C}(L) \geq \varepsilon(X, L,{\rm Sing}(\mathcal{C})).$$
	In general, this inequality is strict. 
	\begin{example}
		In Corollary \ref{prop:star}, we considered a \emph{star} configuration $\mathcal{C} = \{C_{1}, \ldots, C_{d}\}$ with $d \geq 4$ such that $ C_i \in |A|$ for all $i$, where $A$ was a fixed divisor on $X$. We saw that for an ample line bundle $L$  on $X$, if 
		$\frac{dC_1^2}{(d-1)(L\cdot C_1)}L-C_1$ is nef, the Seshadri constant is given by
		$$\varepsilon(X, L,{\rm Sing}(\mathcal{C})) = \frac{L\cdot C_1}{C_1^2(d-1)}.$$
		
		On the other hand, the configurational Seshadri constant is given by
		$$	\varepsilon_\mathcal{C}(L)= \frac{d(L\cdot C_1)}{2t_2}= \frac{L\cdot C_1}{C_1^2(d-1)}.$$
		Thus in this case, both these constants agree.
	\end{example}
	
	\begin{example}
		In Example \ref{KCW}, we first considered a curve arrangement  $\mathcal{C}_n'$ on $X_e$ with $e > n$ obtained as a pullback of the $n$-th Fermat arrangement of lines $\mathcal{L}_n \subset \mathbb{P}^2$  with $n \geq 3.$ This arrangement $\mathcal{C}_n'$ consists of $3n$ curves with $t_n=3e$ and $t_3=en^2.$ Then for the ample line bundle $L=C_0+(e+1)f $ on $X_e$, we obtained the following value of Seshadri constant 
		$$\varepsilon(X_e, L,{\rm Sing}(\mathcal{C}_n')) = \frac{L\cdot C_1}{e(n+1)}=\frac{e+1}{e(n+1)}.$$
		
		On the other hand, the configurational Seshadri constant is given by
		$$	\varepsilon_{\mathcal{C}_n'}(L)= \frac{3n(L\cdot C_1)}{f_1}= \frac{e+1}{e(n+1)}.$$
		Thus, in this case, both these constants agree.
		
		Next, we considered a curve arrangement  $\mathcal{K}'$ on $X_e$ with $e > n$ which is a pullback of the Klein arrangement of lines $\mathcal{K} \subset \mathbb{P}^2.$ The curve arrangement  $\mathcal{K}'$ consists of 21 curves with $t_4=21e$ and $t_3=28e.$ Then by the choice of an ample line bundle $L$ on $X_e$ as in Example \ref{Hirz}, we obtained the following value of Seshadri constant 
		$$\varepsilon(X_e, L,{\rm Sing}(\mathcal{K}')) = \frac{e+1}{8e}.$$
		On the other hand, the configurational Seshadri constant is given by
		$$	\varepsilon_{\mathcal{K}'}(L)= \frac{21(L\cdot C_1)}{f_1}= \frac{e+1}{8e}.$$
		Thus, in this case too, both the constants agree.
		
		We then considered the curve arrangement  $\mathcal{W}'$ on $X_e$ with $e > n$ obtained  as a pullback of the Wiman arrangement of lines $\mathcal{W} \subset \mathbb{P}^2$. The curve arrangement $\mathcal{W}'$ consists of 45 curves with $t_3=120e$, $t_4=45e$ and $t_5=36e.$ Then by the choice of an ample line bundle $L$ on $X_e$ as in Example \ref{Hirz}, we get
		$$\varepsilon(X_e, L,{\rm Sing}(\mathcal{W}')) = \frac{L\cdot C_1}{e(n+1)}=\frac{e+1}{16e}.$$
		
		On the other hand, the configurational Seshadri constant is given by
		$$	\varepsilon_{\mathcal{W}'}(L)= \frac{45(L\cdot C_1)}{f_1}= \frac{e+1}{16e}.$$
		Thus, in this case too, both the constants agree.
		
	\end{example}
We give an example below where the Seshadri constant is strictly smaller than the configurational Seshadri constant. 
\begin{example}
In \cite[Example 3.1]{JP}, it was shown that for a Hirzebruch quasi-pencil $\mathcal{H}$ of $k\geq 4$ lines in $\mathbb{P}^2$ with $t_{k-1}=1$ and $t_2=k-1$, the multi-point Seshadri constant $\varepsilon(\mathbb{P}^2, \mathcal{O}_{\mathbb{P}^2}(1),{\rm Sing}(\mathcal{H}))$ is strictly smaller than its configurational value $\varepsilon_{\mathcal{H}}(\mathcal{O}_{\mathbb{P}^2}(1)).$ We can obtain a configuration $\mathcal{H}'$ on a rational ruled surface $X_e$ with $e=k+1$, as a pullback of the Hirzebruch quasi-pencil $\mathcal{H}$, using the construction mentioned above. The configuration $\mathcal{H}'$ consists of $k \geq 4$ curves say $C_{1}, \ldots, C_{k}$, with $t_{k-1}=e$, $t_2=(k-1)e$ and $ C_i \in |C_0+ef|$ for all $i.$ Let $L=C_0+(e+1)f$ be an ample line bundle on $X_e.$ 

In this case, the configurational Seshadri constant is 
	$$	\varepsilon_{\mathcal{H}'}(L)= \frac{k(L\cdot C_1)}{2e(k-1)+(k-1)e}= \frac{k(e+1)}{3e(k-1)}=\frac{k(k+2)}{3(k^2-1)}.$$

We now claim that 
	$$\varepsilon(X_e, L,{\rm Sing}(\mathcal{H}')) = \frac{e+1}{e(k-1)}=\frac{k+2}{k^2-1}.$$ 
	
	Indeed, we can assume that the points of multiplicity $k-1$ are defined by intersections of $C_1, \cdots, C_{k-1}$. Then the curve $C_k$ gives the Seshadri ratio $\frac{e+1}{e(k-1)}=\frac{k+2}{k^2-1}.$  Suppose that there exists an irreducible and reduced curve $D$, different from each $C_{i}$ for each $i$, having multiplicity $m_{p}(D)$ at each point $p \in \mathcal{P}={\rm Sing}(\mathcal{H}')$ such that
	\begin{equation*}
		\frac{L \cdot D}{\sum_{p \in \mathcal{P}} m_{p}(D)} < \frac{e+1}{e(k-1)}.		
\end{equation*}

Then 
$$D \cdot (C_1+C_k)=2 (D \cdot C_1) \geq \sum_{p \in \mathcal{P}} m_{p}(D) > \frac{e(k-1)(L \cdot D)}{e+1}.$$
Since $C_1$ is linearly equivalent to $L-f, $ we have $$2(D \cdot L) \geq 2(D \cdot C_{1})> \frac{e(k-1)(L \cdot D)}{e+1},$$ i.e., 
$2(k+2) > k^2-1.$ But this is not possible by the choice of $k$.

	This shows that $\varepsilon_{\mathcal{H}'}(L) > \varepsilon(X_e, L,{\rm Sing}(\mathcal{H}')).$
\end{example}

\subsection{Lower bounds on configurational Seshadri constants}\label{bounds}
We now give some lower bounds on configurational Seshadri constants on ruled surfaces and on surfaces of non-negative Kodaira dimension.	We will 
work with arrangements of curves satisfying Assumption \ref{star1}.


	\begin{theorem}\cite[Equation 4.9]{HS} \label{HTIn}
		Let $X$ be a ruled surface over a smooth curve of genus $g \geq 0$ with invariant $e
		\ge 4$.  Let $\mathcal{C} = \{C_{1}, \ldots, C_{d}\}$ be a transversal arrangement of curves
		satisfying Assumption \ref{star1} such that each curve in
		$\mathcal{C}$ is numerically equivalent to $A=aC_0+bf$ with $a > 0$
		and $b \ge ae$. Assume further that all the curves in $\mathcal{C}$ do not meet in a point and that 
  either $a\geq 2$ or if $a=1$, there exists a subset of four curves in $\mathcal{C}$ such that there is no point common to all the four curves. Then we have the following Hirzebruch-type inequality for $\mathcal{C}$:\\
		\begin{equation*}\label{eq:Hirzebruch type ineq}
			t_{2}+\frac{3}{4}t_{3} \geq -16+16g+\sum_{k\geq5}(2k-9)t_{k}+
			d(e(5a^2-2a)-10ab-4ag+4a+4b).\\
		\end{equation*}	
	\end{theorem}

	Using the above result, we give a lower bound on configurational Seshadri constants.	
	\begin{theorem} Let $X$ be a ruled surface over a smooth curve of genus $g \geq 0$ with invariant $e
		\ge 4$.  
		Let
$\mathcal{C} = \{C_{1}, \ldots, C_{d}\}$		
		 be a transversal arrangement of curves satisfying hypothesis of Theorem \ref{HTIn} and let $L$ an ample line bundle on $X$. Then 
		$$\varepsilon_{\mathcal{C}}(
		L) \geq \frac{d(L\cdot C_1)}{8-8g+\frac{9(2ab-a^2e)d^2}{4}+\frac{d(2ae-\frac{a^2e}{2}+ab+4ag-4a-4b)}{2}}.$$
	\end{theorem}
	\begin{proof}
		Our strategy is based on the combinatorial features of $\mathcal{C}$. Let us denote $C = C_{1} + \ldots + C_{d}$. Then we can write 
		\begin{equation}
			\varepsilon_\mathcal{C}(L)= \frac{\sum_{i=1}^{d}L\cdot C_i}{\sum_{P \in \text{Sing}(\mathcal{C}) }\text{mult}_P(\mathcal{C})}=\frac{d(L\cdot C_1)}{f_1}.
		\end{equation}	
		
		Our goal here is to find a reasonable upper-bound on the number $f_{1} = \sum_{r\geq 2} rt_{r}$. In order to do so, we are going to use Theorem \ref{HTIn} and Hirzebruch's inequality, namely
		\begin{equation*}\label{eq:Hirzebruch type ineq}
			t_{2}+\frac{3}{4}t_{3} \geq -16+16g+\sum_{k\geq5}(2k-9)t_{k}+
			d(e(5a^2-2a)-10ab-4ag+4a+4b).\\
		\end{equation*}	
		Simplifying this, we get
		\begin{equation*}\label{eq:4.7}
			16-16g+d(2ae-5a^2e+10ab+4ag-4a-4b)+9f_0 \geq2f_1+4t_2 +t_{4} + \frac{9}{4}t_{3}.
		\end{equation*}
		Since $t_{2}, t_{4},t_{3} \geq 0$ we have
		\begin{equation*}\label{eq:4.7}
			16-16g+d(2ae-5a^2e+10ab+4ag-4a-4b)+9f_0 \geq2f_1.
		\end{equation*}
		and hence
		\begin{equation*}\label{eq:4.7}
			8-8g+\frac{d}{2}(2ae-5a^2e+10ab+4ag-4a-4b)+\frac{9}{2}f_0 \geq f_1.
		\end{equation*}
		Obviously one always has $$d \leq f_{0} \leq (2ab-a^2e) {d \choose 2}$$ which leads to
		\begin{equation*}
			f_1 \leq 8-8g+\frac{9(2ab-a^2e)d^2}{4}+\frac{d(2ae-\frac{a^2e}{2}+ab+4ag-4a-4b)}{2},
		\end{equation*}
		so finally we get
		\begin{equation*}
			\varepsilon_\mathcal{C}(L)=\frac{d(L\cdot C_1)}{f_1} \geq \frac{d(L\cdot C_1)}{8-8g+\frac{9(2ab-a^2e)d^2}{4}+\frac{d(2ae-\frac{a^2e}{2}+ab+4ag-4a-4b)}{2}}.
		\end{equation*}		
	\end{proof}

	We will now consider transversal configurations on surfaces of non-negative Kodaira dimension. We will use the following theorem. 
	\begin{theorem}\cite[Theorem 2.1]{LP1} \label{LP}
		Let $X$ be a smooth complex projective surface with non-negative Kodaira dimension and let $C=C_1 + \cdots + C_d$ be a transversal configuration of smooth curves having $d\geq 2$ irreducible components $C_1, \cdots, C_d.$ Then
		\begin{equation*}
			K_X \cdot C + 4 \sum_{i=1}^{d} (1-g(C_i))-t_2+\sum_{r \geq 3} (r-4)t_r \leq 3c_2(X)-K_X^2.
		\end{equation*}
	\end{theorem}
	We now give a lower bound on configurational Seshadri constants on surfaces of non-negative Kodaira dimension.
 
	\begin{theorem}
		Let $X$ be a smooth complex projective surface with non-negative Kodaira dimension and let $C=C_1 + \cdots + C_d$ be a transversal configuration of smooth curves satisfying Assumption \ref{star1}. Assume further that all the curves in $\mathcal{C}$ do not meet in a point. 
  Let $L$ be an ample line bundle on $X$. Then
		$$\varepsilon_{\mathcal{C}}(
		L) \geq \frac{d(L\cdot C_1)}{3c_2(X)-K_X^2+4C_1^2 {d \choose 2}-K_X \cdot C-4 \sum_{i=1}^{d} (1-g(C_i))}.$$
	\end{theorem}
	\begin{proof}
		As before, we find an upper-bound on the number $f_{1} = \sum_{r\geq 2} rt_{r}$. In order to do so, we are going to use Theorem \ref{LP} and Hirzebruch's inequality, namely
		\begin{equation*}
			K_X \cdot C + 4 \sum_{i=1}^{d} (1-g(C_i))-t_2+\sum_{r \geq 3} (r-4)t_r \leq 3c_2(X)-K_X^2.
		\end{equation*}
		Simplifying, we get 
		\begin{equation*}
			K_X \cdot C + 4 \sum_{i=1}^{d} (1-g(C_i))+t_2+f_1-4f_0\leq 3c_2(X)-K_X^2.
		\end{equation*}	
		Since $t_2 \geq 0,$ we have 
		\begin{equation*}
			K_X \cdot C + 4 \sum_{i=1}^{d} (1-g(C_i))+f_1-4f_0\leq 3c_2(X)-K_X^2
		\end{equation*}	
		and hence 
		\begin{equation*}
			f_1\leq 3c_2(X)-K_X^2+4f_0-K_X \cdot C-4 \sum_{i=1}^{d} (1-g(C_i)).
		\end{equation*}	
		We know that
			$d \leq f_0 \leq C_1^2 {d \choose 2}$. So 
		\begin{equation*}
			f_1\leq 3c_2(X)-K_X^2+4C_1^2 {d \choose 2}-K_X \cdot C-4 \sum_{i=1}^{d} (1-g(C_i)). 
		\end{equation*}	
		Finally we obtain
		\begin{equation*}
			\varepsilon_\mathcal{C}(L)=\frac{d(L\cdot C_1)}{f_1} \geq \frac{d(L\cdot C_1)}{3c_2(X)-K_X^2+4C_1^2 {d \choose 2}-K_X \cdot C-4 \sum_{i=1}^{d} (1-g(C_i))}.
		\end{equation*}		
	\end{proof}

We end with an example of a configuration of lines on a K3 surface in $\mathbb{P}^3$. In this example, the nefness hypothesis of Theorem \ref{main-theorem} is not satisfied, but the conclusion holds. 

\begin{example}\label{K3}
		Let $X \subset \mathbb{P}^3$ be a Fermat quartic defined by the vanishing locus of $x_0^4 + x_1^4 + x_2^4 + x_3^4$, where $x_0, \ldots, x_3$ are the coordinates of $\mathbb{P}^3$. 
		Then $X$ contains the following 48 lines; 
  16 from each of the following three types: 
		\begin{eqnarray*}
			A: \quad x_0 = \alpha x_1 \; \text{and} \; x_2 = \beta x_3, \\
			A': \quad x_0 = \alpha x_2 \; \text{and} \; x_1 = \beta x_3, \\
			A'': \quad x_0 = \alpha x_3 \; \text{and} \; x_1 = \beta x_2,
		\end{eqnarray*}
		where $\alpha, \beta \in \{\zeta, -\zeta, i\zeta, -i\zeta \}$ and $\zeta$ is a primitive eighth root of unity. 

Now let $\mathcal{L} = \{l_1,l_2, \ldots,  l_{48}\}$ be the arrangement consisting of all these 48 lines. 
Intersection behaviour of lines on a Fermat quartic is classically well-known. 
In particular, it is known that the only singular points for the arrangement $\mathcal{L}$ are either double points or quadruple points; \cite[Example 3.3]{Pok1}. It is also not difficult to see that the number of singular points on any line $l$ in $\mathcal{L}$ is 10. Of these, two are quadruple points obtained by intersecting $l$ with lines of the same type. The remaining eight points are double points obtained by intersecting $l$ with lines of different types. 


The hyperplane section $X \cap V(x_0 - \zeta x_1)$ is the union of the four lines of type $A$ in $\mathcal{L}$ 
given by  
$x_0 = \zeta x_1 \; \text{and} \; x_2 = \beta_j x_3$
for  $1\leq j \leq 4$. Consequently, the sum of these four lines is linearly equivalent to $\mathcal{O}_X(1)$. Similarly, the sum of the four lines 
$x_0 = \alpha_i x_1 \; \text{and} \; x_2 = \zeta x_3$, 
for $1\leq i \leq 4$, 
is linearly equivalent to $\mathcal{O}_X(1)$. Thus we note that the line 
$x_0 = \zeta x_1 \; \text{and} \; x_2 = \zeta x_3$ is in two different sets of 4 lines each so that the sum of lines in each set is linearly equivalent to $\mathcal{O}_X(1)$.
Using a similar argument, one can see that there are 24 sets of 4 lines each such that each line in $\mathcal{L}$ appears in \textit{exactly} two of the sets and the sum of the 4 lines in each set is linearly equivalent to $\mathcal{O}_X(1)$.



		
		We now compute the multi-point Seshadri constant of $\mathcal{O}_X(1)$ at the singular locus of 
		the arrangement $\mathcal{L}$.  This arrangement has 216 singular points. So we have 
		$$\varepsilon(X, \mathcal{O}_X(1), \text{Sing}(\mathcal{L})) \le \sqrt{\frac{4}{216}} \sim 0.136.$$
In fact, we claim that
		\begin{equation}\label{2.22}
		\varepsilon(X, \mathcal{O}_X(1), \text{Sing}(\mathcal{L})) = \frac{1}{10}.
	\end{equation}

Note that Theorem \ref{main-theorem} predicts that the Seshadri constant is indeed $\frac{1}{10}$, though  it cannot be applied here. This is because 
$d=48, b_i = 10$ and $L\cdot l_i = 1$ for every $i$. So the divisor which appears in the statement of Theorem \ref{main-theorem}
is not nef, as can be seen by intersecting with $l_i$.

  To prove \eqref{2.22}, we first show that the Seshadri ratio computed by a line in $\mathcal{L}$ is $\frac{1}{10}$. 
   As mentioned above, every line in $\mathcal{L}$ contains exactly 10 singular points of $\mathcal{L}$. 
   Note that each line $l$ in $\mathcal{L}$ has degree 1 in $\mathbb{P}^3$.
   So for every line $l \in \mathcal{L}$, 
		\[
		\frac{\mathcal{O}_X(1)\cdot l}{ \sum\limits_{p \in\text{Sing}(\mathcal{L})} \text{mult}_{p} \,l} 
		= \frac{1}{10}.
		\]
  
		Now let $C \subset X$ be any reduced and irreducible curve such that $C \notin \mathcal{L}$ and which passes through some point in the singular locus of 
		$\mathcal{L}$. We will show that the Seshadri ratio given by $C$ is strictly bigger than $\frac{1}{10}$. 
  
  By B\'ezout's theorem, we have the following inequality for every line $l \in \mathcal{L}$:
		$$2 C \cdot l \ge 2\left( \sum_{p \in l \,\cap \,\text{Sing} (\mathcal{L})} \text{mult}_p C \right).$$
  
We now sum such inequalities over all the 48 lines and then group the lines into sets of four lines as described above, so that the sum of the four lines in each set is linearly equivalent to $\mathcal{O}_X(1)$.  Therefore we get
               \begin{eqnarray*}
                \sum\limits_{i=1}^{48} 2 C \cdot l_i &\ge& \sum\limits_{i=1}^{48} 2\left( \sum_{p \in l_i \,\cap \,\text{Sing} (\mathcal{L})} \text{mult}_p C \right) \\
                \Rightarrow  \sum\limits_{i=1}^{24} C \cdot \mathcal{O}_X(1) &\ge&  2\left( \sum\limits_{i=1}^{48}\sum_{p \in l_i \,\cap \,\text{Sing} (\mathcal{L})} \text{mult}_p C \right).
               \end{eqnarray*}

Since each double point appears in two lines and each quadruple point appears in four lines, we obtain 
		\begin{eqnarray*}
				24 \text{deg}(C) &\geq& 2 \left(2\sum\limits_{p_i \in \text{Sing}_2(\mathcal{L})} \text{mult}_{p_i} C ~+~ 4\sum\limits_{q_j \in \text{Sing}_4(\mathcal{L})} \text{mult}_{q_j}C \right) \\
		 &\geq& 4 \sum\limits_{p \in\text{Sing}(\mathcal{L})} \text{mult}_{p}C.
		\end{eqnarray*}
		%

Here $\text{Sing}_2(\mathcal{L})$ and $\text{Sing}_4(\mathcal{L})$ denote the set of double points and the set of quadruple points of the arrangement $\mathcal{L}$, respectively. Therefore we get 
		\[
		\frac{\mathcal{O}_X(1) \cdot C}{\sum\limits_{p \in \text{Sing}(\mathcal{L})} \text{mult}_{p}C} = \frac{\text{deg}(C)}{ \sum\limits_{p \in\text{Sing}(\mathcal{L})} \text{mult}_{p}C} \geq \frac{1}{6} > \frac{1}{10}.
		\]
		This completes the proof of the equality \eqref{2.22}.

	\end{example}

 \subsection*{Acknowledgements} The first author is partially supported by a grant from Infosys Foundation. 
	
\subsection*{Declaration of interests}
The authors declare that they have no known competing financial interests or personal relationships
that could have appeared to influence the work reported in this paper.

	\bibliographystyle{plain}

\begin{thebibliography}{100}
	\bibitem{B} T. Bauer, Seshadri constants on algebraic surfaces, Math. Ann. \textbf{313} (1999), no.3, 547--583.
	\bibitem{CV} E. Carlini, A. Van Tuyl, Star configuration points and generic plane curves, Proc. Amer. Math. Soc. \textbf{139} (2011), no.12, 4181--4192.
	\bibitem{E} S. Eterovi\'c, Logarithmic Chern slopes of arrangements of rational sections in Hirzebruch surfaces,
		Master Thesis, Pontificia Universidad Cat\'olica de Chile, Santiago 2015.
		\bibitem{GHM} A. V. Geramita, B. Harbourne, J. Migliore, Star configurations in $\mathbb{P}^n$, J. Algebra, \textbf{376} (2013), 279--299.
		\bibitem{HM} K. Hanumanthu, A. Mukhopadhyay, {\it Multi-point Seshadri constants on ruled surfaces}, Proc. Amer. Math. Soc. \textbf{145} (2017), no. 12, 5145--5155.
		\bibitem{HS} K. Hanumanthu, A. Subramaniam, Bounded negativity and Harbourne constants on ruled surfaces, Manuscripta Math. \textbf{164} (2021) 431--454. 
		\bibitem{HR} B. Harbourne, J. Ro\'e,
Discrete behavior of Seshadri constants on surfaces, J. Pure Appl. Algebra \textbf{212} (2008), no.3, 616--627.
		\bibitem{Har} R. Hartshorne, {\it Algebraic geometry},
		Springer-Verlag, New York, 1977.
		\bibitem{JP}M. Janasz and P. Pokora, On Seshadri constants and point-curve configurations, Electron. Res. Arch., \textbf{28(2)} (2020)  795--805.

		
			\bibitem{LP1} R. Laface and P. Pokora, 
		Local negativity of surfaces with non-negative Kodaira dimension and transversal configurations of curves, 
Glasg. Math. J. \textbf{62} (2020), no.1, 123--135.	
		\bibitem{Nag} M. Nagata, On the 14-th problem of Hilbert, Am. J. Math. \textbf{81} (1959), 766--772.
\bibitem{Pok1}	P. Pokora, Harbourne constants and arrangements of lines on smooth hypersurfaces in
$\mathbb{P}^3$,  Taiwanese J. Math., \textbf{20}, No. 1 (2016), 25--31.  
		
		\bibitem{Pok} P. Pokora, Seshadri constants and special configurations of points in the projective plane,
		Rocky Mountain J. Math. \textbf{49(3)} (2019) 963 -- 978.
		\bibitem{PRS} P. Pokora, X. Roulleau and T. Szemberg, Bounded negativity, Harbourne constants and transversal
		arrangements of curves, Ann. Inst. Fourier (Grenoble) \textbf{67} (2017), no. 6, 2719–2735.
		\bibitem{PR} P. K. Roy,
Some results on Seshadri constants on surfaces of general type,  Eur. J. Math. \textbf{6} (2020), no.4, 1176--1190.
		\bibitem{RR} J. Ro\'e, J. Ross, 
An inequality between multipoint Seshadri constants, Geom. Dedicata \textbf{140} (2009), 175--181.
		\bibitem{SS} W. Syzdek, T. Szemberg, Seshadri fibrations of algebraic surfaces, Math. Nachr. \textbf{283} (2010),
no. 6, 902--908.
\bibitem{U} G. A. Urz\'ua,  Arrangements of curves and algebraic surfaces,  Thesis (Ph.D.) University of
Michigan, 2008, 166 pages.
	\end{thebibliography}

\end{document}